\documentclass[11pt]{amsart}
\usepackage[utf8]{inputenc}
\usepackage[T2A,T1]{fontenc}
\usepackage{amsmath,amssymb,amsfonts,amsthm}
\usepackage{mathtools}
\usepackage{mathrsfs}
\usepackage{dsfont}
\usepackage{enumitem}
\usepackage{multicol}
\usepackage{graphicx}
\usepackage{float}
\usepackage{subcaption}
\usepackage{rotating}
\usepackage{lscape}
\usepackage[table]{xcolor}
\definecolor{lightgray}{gray}{0.95}
\usepackage[hidelinks,hypertexnames=false]{hyperref}
\usepackage{geometry}
\usepackage[russian,english]{babel}

\DeclareRobustCommand{\ru}[1]{\foreignlanguage{russian}{#1}}

\usepackage{algorithm}
\usepackage{algpseudocode}
\usepackage[most]{tcolorbox}

\definecolor{algogray}{gray}{0.96}

\tcbset{
  algobox/.style={
    colback=algogray,
    colframe=gray!50,
    boxrule=0.7pt,
    arc=2pt,
    left=7pt,
    right=7pt,
    top=7pt,
    bottom=7pt,
    enhanced
  }
}

\theoremstyle{plain}
\newtheorem{theorem}{Theorem}[section]

\newtheorem{proposition}[theorem]{Proposition}

\newtheorem{definition}[theorem]{Definition}

\newcommand{\dps}{\displaystyle}

\setlist{itemsep=2pt, topsep=7pt, parsep=2pt}

\title[The Manufacture of Examples]{The Manufacture of Examples: the \((-1)\)-classical orthogonal polynomials}

\author{K. Castillo}
\address{CMUC, Department of Mathematics, University of Coimbra, 3000-143 Coimbra,
Portugal}
\email{kenier@mat.uc.pt}

\author{G. Gordillo-N\'u\~nez}
\address{CMUC, Department of Mathematics, University of Coimbra, 3000-143 Coimbra,
Portugal}
\email{up202310693@up.pt}

\subjclass[2020]{42C05, 46A13}

\keywords{\((-1)\)-classical orthogonal polynomials, alternating maps,
orthogonality functionals, Christoffel transforms, Geronimus transforms,
Bannai--Ito polynomials, big \((-1)\)-Jacobi polynomials,
\((-1)\)-Meixner--Pollaczek polynomials}

\date{\today}

\begin{document}

\begin{abstract}
The big \((-1)\)-Jacobi, Bannai--Ito, and
\((-1)\)-Meixner--Pollaczek polynomials have appeared in the mathematical
physics literature through differential- or difference-reflection eigenvalue
problems.
We show that these three families are governed by a single alternating
mechanism, formulated at the level of orthogonality functionals, transpose
operators, and structural identities. Starting from an ordinary orthogonal
polynomial sequence in the quadratic variable, the resulting algorithm
constructs the direct alternating family and transfers its moments, squared
norms, and recurrence coefficients. One regular Geronimus inversion then
produces the big \((-1)\)-Jacobi and Bannai--Ito families, whereas the
\((-1)\)-Meixner--Pollaczek family is obtained directly.
\end{abstract}

\maketitle

\section{Introduction}

When it exists, a positive-measure representation is one concrete realisation
of orthogonality with respect to a regular continuous linear functional; it is
not its definition. Such a representation may supply an accompanying
Hilbert-space structure, but positivity is an additional hypothesis and is not
required for the polynomial theory considered here. The appropriate general
framework is functional-analytic: duality on the locally convex polynomial
space allows the three-term recurrence, polynomial spectral transformations,
and classical structural equations to be formulated over \(\mathbb C\) without
an extraneous positivity restriction. The distinction is indispensable here:
the bilinear form is
\[
(p,q)_{\mathbf u}=\langle\mathbf u,pq\rangle,
\]
and regularity requires its restriction to each polynomial space to be
non-degenerate; positivity is an additional hypothesis.

Orthogonal polynomial sequences described as ``new classical families''
continue to appear in the literature. The issue depends, of course, on the
meaning assigned to classicality. The tabular presentation of classical
orthogonal polynomials in the \emph{Digital Library of Mathematical Functions}
(DLMF), published by the National Institute of Standards and Technology
(NIST), and the organisational charts of the usual Askey scheme are valuable
reference tools; they are not a classification of the regular functional
solutions of the structural equation. Treating them as exhaustive for that
purpose gives too narrow a notion of classicality; compare
\cite[Sections~18.3, 18.21, and~18.28(xi)]{DLMF} with
\cite[Definition~2.5, Theorem~4.1, and Section~7]{CG26a}. In particular,
\cite[Section~18.28(xi)]{DLMF} records \(q\downarrow-1\) limits and mentions the
Bannai--Ito polynomials, but it neither formulates the alternating calculus nor
classifies its regular functional solutions. It therefore cannot by itself
decide whether a family is new in the structural sense used here.

The classification of admissible maps in \cite[Theorem~5.1]{C26} isolates the
alternating case as a distinct geometry. The big \((-1)\)-Jacobi,
Bannai--Ito, and \((-1)\)-Meixner--Pollaczek polynomials are representative
examples. They have been studied through Dunkl-type eigenvalue problems
involving reflection, and through limiting and spectral-transform arguments. The
alternating mechanism used below is developed in \cite[Section~8]{C26}; its
relation to the framework of Nikiforov and Uvarov \cite{NU83,NU84} is discussed
in \cite[Sections~5 and~9]{C26}. The identification of the three families is
established in the examples below.

At the level of admissible-map geometries, the quadratic and
\(q\)-exponential cases correspond to cases~\emph{(i)} and~\emph{(iii)} of
\cite[Theorem~5.1]{C26}. The families considered here belong instead to
case~\emph{(ii)}, the alternating case. The \(q\downarrow-1\) material in the
usual scheme records a limiting relation; it does not formulate the alternating
calculus as a third admissible-map geometry. That structural distinction is the
one used below.

We make the argument independently readable by recalling all notation and
definitions needed below. Starting from an orthogonal polynomial sequence, the
alternating construction is followed, when necessary, by one regular inversion.
The three families above then become particular instances of the same
mechanism. Thus a systematic procedure replaces separate ad hoc constructions
and exposes the structural reason for their common form.

The purpose is structural and classificatory: the construction identifies the
common functional mechanism, transfers moments, squared norms, and
recurrence coefficients from the quadratic variable, and gives a recognition
test for further proposed \((-1)\)-families of this form.

The material that follows uses \cite[Section~8]{C26}; we recall the notation
and the precise statements required for the present construction. Let
\[
\mathbb N=\{0,1,2,\ldots\},
\quad
\mathbb N^\times=\mathbb N\setminus\{0\}.
\]
Let
\(h\in\mathbb C^\times\), and let \(U\subseteq\mathbb C\) be a non-empty set
such that
\[
U\pm \frac h2\subseteq U.
\]
Such a set will be called \emph{half-step-invariant}. A simple example is
\[
U=h\mathbb Q,
\]
which is stable under translations by \(\pm h/2\). We consider the normalised
alternating map
\[
X:U\to\mathbb C,
\quad
X(s)=e^{\pi i s/h}.
\]
In this case,
\[
X\!\left(s+\frac h2\right)=iX(s),
\quad
X\!\left(s-\frac h2\right)=-\,iX(s),
\]
so the two half-step neighbours of \(X(s)\) are obtained by multiplication by
\(\pm i\). Moreover,
\(
X(h\mathbb Q)
\)
is infinite, since the values
\[
X\!\left(\frac{h}{2r+1}\right)=e^{\pi i/(2r+1)},
\quad r\in\mathbb N,
\]
are distinct. Whenever \(U\) is half-step-invariant and \(X(U)\) is infinite,
we shall say that the normalised alternating map \(X\) is \emph{admissible} on
\(U\). The formulae that follow are first obtained pointwise on \(U\), but
admissibility ensures that they can be read as identities in the polynomial
variable \(X\). Indeed, since \(X(U)\) is infinite, any two polynomials that
agree on \(X(U)\) must agree identically. Without this, one would have only
relations valid on a finite image, and hence no intrinsic polynomial calculus.

Moreover,
\[
X(s+h)=-X(s),
\quad
X(s+2h)=X(s).
\]
The adjective \emph{alternating} refers to this full-step sign change, which
here follows from the preceding half-step identities. Once the normalisation
has been fixed, the induced polynomial calculus depends only on the
neighbouring values \(iX\) and \(-iX\). The data \(U,h\) establish its
admissibility and therefore do not reappear among the inputs of the
reconstruction.

If \(\mathcal P\) denotes the vector space of complex polynomials in one
variable, then every \(p\in\mathcal P\) admits a unique decomposition of the
form
\[
p(x)=a(-x^2)+x\,b(-x^2),
\quad
a,b\in\mathcal P.
\]
Substituting \(iX\) and \(-iX\) into this decomposition gives
\[
p(iX)=a(X^2)+iX\,b(X^2),
\quad
p(-iX)=a(X^2)-iX\,b(X^2),
\]
and hence
\[
\frac{p(iX)-p(-iX)}{2iX}=b(X^2),
\quad
\frac{p(iX)+p(-iX)}{2}=a(X^2).
\]
By admissibility, these are not merely pointwise identities on the image of
\(U\), but polynomial identities. Thus the alternating divided-difference and
averaging procedures separate the odd and even components of \(p\) relative to
the quadratic substitution \(x\mapsto -x^2\). This is the structural feature
from which the subsequent analysis emerges.

Accordingly, to the fixed admissible normalised alternating map \(X\) on
\(U\) we attach linear operators
\[
D,S:\mathcal P\to\mathcal P
\]
defined by
\[
(Dp)(X)=\frac{p(iX)-p(-iX)}{2iX},
\quad
(Sp)(X)=\frac{p(iX)+p(-iX)}{2},
\quad p\in\mathcal P.
\]
By admissibility, these identities define \(Dp\) and \(Sp\) uniquely as
polynomials in the variable \(X\). In particular,
\[
(Dp)(X)=b(X^2),
\quad
(Sp)(X)=a(X^2).
\]
We regard \(\mathcal P\) as a locally convex space endowed with the strict
countable inductive-limit topology arising from the filtration by the
finite-dimensional spaces \(\mathcal P_n\) of polynomials of degree at most
\(n\). We write \(\mathcal P'\) for its continuous dual, endowed with the weak
topology \(\sigma(\mathcal P',\mathcal P)\). For \(p_0\in\mathcal P\) and
\(\mathbf u\in\mathcal P'\), multiplication is the transpose action defined by
\[
\langle p_0\mathbf u,p\rangle=\langle\mathbf u,p_0p\rangle,
\quad p\in\mathcal P.
\]
The preceding formulae give
\[
D(\mathcal P_n)\subseteq\mathcal P_n,
\quad
S(\mathcal P_n)\subseteq\mathcal P_n,
\quad n\in\mathbb N.
\]
Their restrictions to each finite-dimensional space \(\mathcal P_n\) are
therefore continuous, and hence so are \(D\) and \(S\) on the inductive limit.
Their transposes are consequently well defined, and we obtain weakly continuous
operators
\[
\mathbf D,\mathbf S:\mathcal P'\to\mathcal P'
\]
characterised by
\[
\langle \mathbf D\mathbf u,p\rangle=-\langle \mathbf u,Dp\rangle,
\quad
\langle \mathbf S\mathbf u,p\rangle=\langle \mathbf u,Sp\rangle,
\]
for every \(\mathbf u\in\mathcal P'\) and every \(p\in\mathcal P\); see
\cite[Definitions~6.3 and~6.4]{C26}. For the present note we restrict the
notions in \cite[Definitions~6.2 and~6.6]{C26} to the infinite regular
normalised alternating setting.

\begin{definition}[Orthogonality in the infinite regular setting]
Let \(\mathbf u\in\mathcal P'\). A polynomial sequence
\(
(P_n)_{n\in\mathbb N}
\)
with
\[
\deg P_n=n,
\quad n\in\mathbb N,
\]
is said to be \emph{orthogonal with respect to} \(\mathbf u\) if
\[
\langle \mathbf u,P_nP_m\rangle=0
\]
for every \(m,n\in\mathbb N\) with \(m<n\), and
\[
\langle \mathbf u,P_n^2\rangle\neq0
\]
for every \(n\in\mathbb N\). We say that \(\mathbf u\) is \emph{regular} if
there exists an orthogonal polynomial sequence
\(
(P_n)_{n\in\mathbb N}
\)
with respect to \(\mathbf u\).
\end{definition}

Within this infinite regular setting, classicality takes the following form.

\begin{definition}[\((-1)\)-classicality]
\label{def:classical-normalised-alternating}
Fix \(h\in\mathbb C^\times\), let \(U\subseteq\mathbb C\) be
half-step-invariant, and let
\[
X:U\to\mathbb C,
\quad
X(s)=e^{\pi i s/h},
\]
be an admissible normalised alternating map on \(U\). Let
\[
D,S:\mathcal P\to\mathcal P
\]
be the divided-difference and averaging operators associated with this fixed
map \(X\), and let
\[
\mathbf D,\mathbf S:\mathcal P'\to\mathcal P'
\]
denote their transposes. A functional \(\mathbf u\in\mathcal P'\) is said to be
\emph{classical for the normalised alternating map \(X\)} if it is regular and
if there exist polynomials \(\phi,\psi\in\mathcal P\), not both identically
zero, where \(\phi\) has degree at most \(2\) and \(\psi\) has degree at most
\(1\), such that
\[
\mathbf D(\phi\,\mathbf u)=\mathbf S(\psi\,\mathbf u)
\]
in \(\mathcal P'\).
\end{definition}

The direct reconstruction below concerns structural pairs for which
\(\deg\psi=1\). After multiplying the structural equation by a non-zero
constant, there is a unique \(\tau\in\mathbb C\) such that
\[
\psi(x)=x-\tau.
\]
A non-zero constant \(\psi\) is incompatible with regularity. Indeed, pairing
the structural equation with \(1\) gives
\(0=\langle\mathbf u,\psi\rangle\), and hence
\(\langle\mathbf u,1\rangle=0\). A structural pair with \(\psi=0\) will
reappear below after a regular inversion of a functional.

Let
\[
\sigma:\mathcal P\to\mathcal P,
\quad
(\sigma p)(x)=p(x^2),
\]
be the quadratic substitution operator. Since
\[
\sigma(\mathcal P_n)\subseteq\mathcal P_{2n},
\quad n\in\mathbb N,
\]
the operator \(\sigma\) is continuous. Its transpose is therefore well defined,
and we obtain a weakly continuous operator
\[
\boldsymbol{\sigma}:\mathcal P'\to\mathcal P'
\]
characterised by
\[
\langle \boldsymbol{\sigma}\mathbf u,p\rangle=\langle \mathbf u,\sigma p\rangle,
\]
for every \(\mathbf u\in\mathcal P'\) and every \(p\in\mathcal P\); see
\cite[Definition~8.2]{C26}. If \(\mathbf u\) is regular and admits a structural
pair with \(\deg\psi=1\), \cite[Proposition~8.3]{C26} yields
\[
\mathbf S\bigl((x-\tau)\mathbf u\bigr)=0.
\]
Equivalently,
\[
\langle \mathbf u,(x-\tau)p(x^2)\rangle=0
\]
for every \(p\in\mathcal P\). Thus, in the normalised alternating case, a
non-degenerate structural pair is governed by a first-order annihilation
condition in the quadratic variable. For a fixed regular output, the
parameter is recovered from
\[
\tau=
\frac{\langle\mathbf u,x\rangle}{\langle\mathbf u,1\rangle}.
\]
Thus \(\tau\) is an input in the forward construction and a uniquely
recoverable invariant in the inverse problem. This annihilation condition is
the core of the direct reconstruction; regularity is controlled by the
non-vanishing condition in the proposition below. It reduces directly
constructions which otherwise appear as lengthy specialised arguments.

The following proposition is the infinite reconstruction contained in
\cite[Theorem~8.6]{C26}. We include the direct verification in order to fix the
normalisation used by the algorithm and by the recurrence formulae below.

\begin{proposition}\label{prop:alternating-construction-algorithm}
Let \(\mathbf v\in\mathcal P'\) admit a monic orthogonal polynomial sequence
\((R_n)_{n\in\mathbb N}\), and let \(\tau\in\mathbb C\). Assume that
\[
R_n(\tau^2)\neq0
\]
for every \(n\in\mathbb N\). Then there exists a unique functional
\(\mathbf u\in\mathcal P'\) such that
\[
\boldsymbol{\sigma}\mathbf u=\mathbf v,
\quad
\langle \mathbf u,(x-\tau)p(x^2)\rangle=0
\]
for every \(p\in\mathcal P\). Moreover, if, for \(n\in\mathbb N\),
\[
S_n(x)=
\frac{
R_{n+1}(x)-\dfrac{R_{n+1}(\tau^2)}{R_n(\tau^2)}\,R_n(x)
}{
x-\tau^2
},
\]
and \((P_n)_{n\in\mathbb N}\) is defined by
\[
P_{2n}(x)=R_n(x^2),
\quad
P_{2n+1}(x)=(x-\tau)S_n(x^2),
\]
then \((P_n)_{n\in\mathbb N}\) is the monic orthogonal polynomial sequence
with respect to \(\mathbf u\). In addition, \(\mathbf u\) is classical for the
normalised alternating map with
\[
\phi=0,
\quad
\psi=x-\tau.
\]
\end{proposition}
\begin{proof}
Every polynomial \(f\in\mathcal P\) can be written uniquely in the form
\[
f(x)=p(x^2)+(x-\tau)q(x^2),
\quad p,q\in\mathcal P.
\]
We define \(\mathbf u\in\mathcal P'\) by
\[
\langle \mathbf u,p(x^2)+(x-\tau)q(x^2)\rangle
=
\langle \mathbf v,p\rangle.
\]
This is well defined by the uniqueness of the preceding decomposition. It
belongs to \(\mathcal P'\), since its restriction to each finite-dimensional
space \(\mathcal P_n\) is continuous and \(\mathcal P\) carries their
inductive-limit topology. It
immediately gives
\[
\boldsymbol{\sigma}\mathbf u=\mathbf v,
\quad
\langle \mathbf u,(x-\tau)q(x^2)\rangle=0,
\]
and uniqueness follows from the same decomposition. It remains to identify the
orthogonal polynomials. By the Christoffel formula, the condition
\(R_n(\tau^2)\neq0\) implies that, for \(n\in\mathbb N\),
\[
S_n(y)=
\frac{
R_{n+1}(y)-\dfrac{R_{n+1}(\tau^2)}{R_n(\tau^2)}R_n(y)
}{
y-\tau^2
}
\]
is the monic orthogonal polynomial sequence with respect to
\((y-\tau^2)\mathbf v\). Now set, for \(n\in\mathbb N\),
\[
P_{2n}(x)=R_n(x^2),
\quad
P_{2n+1}(x)=(x-\tau)S_n(x^2).
\]
These polynomials are monic and have degrees \(2n\) and \(2n+1\),
respectively. We check orthogonality. First,
\[
\langle \mathbf u,P_{2n}P_{2m+1}\rangle
=
\langle \mathbf u,
R_n(x^2)(x-\tau)S_m(x^2)
\rangle
=0
\]
by the defining annihilation condition for \(\mathbf u\). Secondly,
\[
\langle \mathbf u,P_{2n}P_{2m}\rangle
=
\langle \mathbf u,R_n(x^2)R_m(x^2)\rangle
=
\langle \mathbf v,R_nR_m\rangle,
\]
which vanishes for \(n\ne m\). Finally,
\[
\langle \mathbf u,P_{2n+1}P_{2m+1}\rangle
=
\langle \mathbf u,
(x-\tau)^2S_n(x^2)S_m(x^2)
\rangle.
\]
Since
\[
(x-\tau)^2=(x^2-\tau^2)-2\tau(x-\tau),
\]
and since \(\mathbf u\) annihilates
\((x-\tau)S_n(x^2)S_m(x^2)\), we obtain
\[
\langle \mathbf u,P_{2n+1}P_{2m+1}\rangle
=
\langle \mathbf u,(x^2-\tau^2)S_n(x^2)S_m(x^2)\rangle.
\]
Hence
\[
\langle \mathbf u,P_{2n+1}P_{2m+1}\rangle
=
\langle \mathbf v,(y-\tau^2)S_n(y)S_m(y)\rangle,
\]
which vanishes for \(n\ne m\), since \((S_n)_{n\in\mathbb N}\) is orthogonal
with respect to \((y-\tau^2)\mathbf v\). Put
\[
h_n=\langle\mathbf v,R_n^2\rangle,
\quad
\rho_n=\frac{R_{n+1}(\tau^2)}{R_n(\tau^2)},
\quad n\in\mathbb N.
\]
The defining relation for \(S_n\) gives
\[
\langle (y-\tau^2)\mathbf v,S_n^2\rangle=-\rho_nh_n.
\]
Consequently,
\[
\langle\mathbf u,P_{2n}^2\rangle=h_n,
\quad
\langle\mathbf u,P_{2n+1}^2\rangle=-\rho_nh_n,
\quad n\in\mathbb N.
\]
Both quantities are non-zero. Thus \((P_n)_{n\in\mathbb N}\) is the monic
orthogonal polynomial sequence with respect to \(\mathbf u\). Its defining
annihilation identity and \cite[Proposition~8.3]{C26} give the stated
structural pair.
\end{proof}

The functional and recurrence data of the direct output are equally explicit.
If
\[
m_k=\langle\mathbf v,x^k\rangle,
\quad k\in\mathbb N,
\]
then
\[
\langle\mathbf u,x^{2k}\rangle=m_k,
\quad
\langle\mathbf u,x^{2k+1}\rangle=\tau m_k,
\quad k\in\mathbb N.
\]
Suppose that the input sequence satisfies
\[
R_{n+1}(x)=(x-\beta_n)R_n(x)-\gamma_nR_{n-1}(x),
\quad n\in\mathbb N,
\]
where \(R_{-1}=0\), \(R_0=1\), \(\gamma_0=0\), and
\[
\gamma_n\neq0,
\quad n\in\mathbb N^\times.
\]
Put
\[
\rho_n=\frac{R_{n+1}(\tau^2)}{R_n(\tau^2)},
\quad n\in\mathbb N.
\]
Then
\[
\rho_0=\tau^2-\beta_0,
\quad
\rho_n=\tau^2-\beta_n-\frac{\gamma_n}{\rho_{n-1}},
\quad n\in\mathbb N^\times.
\]
The output sequence satisfies
\[
P_{k+1}(x)=(x-B_k)P_k(x)-C_kP_{k-1}(x),
\quad k\in\mathbb N,
\]
where \(P_{-1}=0\), \(P_0=1\), \(C_0=0\), and
\[
B_{2n}=\tau,
\quad
B_{2n+1}=-\tau,
\quad n\in\mathbb N,
\]
\[
C_{2n+1}=-\rho_n,
\quad n\in\mathbb N,
\quad
C_{2n}=-\frac{\gamma_n}{\rho_{n-1}},
\quad n\in\mathbb N^\times.
\]
Indeed, if
\[
H_k=\langle\mathbf u,P_k^2\rangle,
\quad k\in\mathbb N,
\]
then \(C_k=H_k/H_{k-1}\) for \(k\in\mathbb N^\times\). Since
\(h_n=\gamma_nh_{n-1}\) for \(n\in\mathbb N^\times\), the norm identities in
the proof give the two formulae for \(C_k\). The annihilation condition gives
\[
\langle\mathbf u,(x-\tau)P_{2n}^2\rangle=0,
\quad
\langle\mathbf u,(x+\tau)P_{2n+1}^2\rangle=0,
\quad n\in\mathbb N,
\]
where the second identity follows from
\[
(x+\tau)P_{2n+1}(x)^2
=(x-\tau)(x^2-\tau^2)S_n(x^2)^2.
\]
Multiplying the monic recurrence by \(P_k\) and applying \(\mathbf u\) gives
\[
B_k=\frac{\langle\mathbf u,xP_k^2\rangle}{H_k},
\quad k\in\mathbb N,
\]
and therefore \(B_{2n}=\tau\) and \(B_{2n+1}=-\tau\).

Thus every recurrence coefficient of the output is obtained recursively from
the input recurrence coefficients using only rational operations. These
identities also give an exact recurrence-level recognition test. A regular
monic sequence satisfying the displayed recurrence is a direct output of
Proposition~\ref{prop:alternating-construction-algorithm} if and only if there
is a \(\tau\in\mathbb C\) such that
\[
B_{2n}=\tau,
\quad
B_{2n+1}=-\tau,
\quad n\in\mathbb N.
\]
In that event the input recurrence is recovered uniquely from
\[
\beta_0=\tau^2+C_1,
\quad
\beta_n=\tau^2+C_{2n}+C_{2n+1},
\quad
\gamma_n=C_{2n-1}C_{2n},
\quad n\in\mathbb N^\times.
\]
Necessity follows from the preceding formulae. Conversely, put
\(\rho_n=-C_{2n+1}\). Since the candidate sequence is regular, every
\(C_k\), with \(k\in\mathbb N^\times\), is non-zero. Hence the recovered
coefficients satisfy \(\gamma_n\neq0\), and their three-term recurrence defines
a regular monic input sequence \((R_n)_{n\in\mathbb N}\). Induction in its
evaluation recurrence gives
\[
\frac{R_{n+1}(\tau^2)}{R_n(\tau^2)}=\rho_n=-C_{2n+1},
\quad n\in\mathbb N.
\]
Thus the required values are non-zero. The direct construction has the same
coefficients \(B_k,C_k\) as the candidate; uniqueness of the monic recurrence
therefore identifies the two sequences.

For the big
\((-1)\)-Jacobi intermediate family, the corresponding polynomial and
recurrence formulae already appear in \cite[(3.21)--(3.24)]{VZ12}; the
present formulation gives their functional meaning and applies uniformly.

Proposition~\ref{prop:alternating-construction-algorithm} clarifies the scope
of the continuous \(-1\) scheme described in \cite{PVZ24}. That scheme
organises a substantial collection of families by means of \(q\to -1\) limits,
specialisations, and explicit recurrence calculations. For the big
\((-1)\)-Jacobi and \((-1)\)-Meixner--Pollaczek entries in that scheme, as
well as for the Bannai--Ito family of \cite{TVZ12} treated here, however, the
direct outputs arise systematically from ordinary orthogonal polynomial
systems in the quadratic variable; the big \((-1)\)-Jacobi and Bannai--Ito
families then follow by one regular Geronimus inversion. For these three
families, the role of \(q\to -1\) is not accidental, but neither is it the
source of the phenomenon. It is a limiting manifestation of the alternating
case, whose intrinsic structure is already determined by the quadratic
decomposition of polynomials.

For ease of use, the forward construction and the inverse recurrence test are
recorded separately in Algorithms~\ref{alg:alternating-pullback}
and~\ref{alg:alternating-recognition}.

\begin{algorithm}[t]
\caption{The alternating reconstruction}
\label{alg:alternating-pullback}
\begin{tcolorbox}[algobox]
\begin{algorithmic}[1]
\Require A regular functional \(\mathbf v\in\mathcal P'\), its monic
orthogonal polynomial sequence
\[
(R_n)_{n\in\mathbb N},
\]
and a parameter \(\tau\in\mathbb C\) such that
\[
R_n(\tau^2)\neq0,
\quad n\in\mathbb N.
\]

\Ensure The monic orthogonal polynomial sequence
\[
(P_n)_{n\in\mathbb N}
\]
associated with the functional \(\mathbf u\) determined by
\[
\boldsymbol{\sigma}\mathbf u=\mathbf v,
\quad
\langle \mathbf u,(x-\tau)p(x^2)\rangle=0
\]
for every \(p\in\mathcal P\).
\State Construct the monic orthogonal polynomial sequence
\[
(S_n)_{n\in\mathbb N}
\]
by
\[
S_n(y)=
\frac{
R_{n+1}(y)-\dfrac{R_{n+1}(\tau^2)}{R_n(\tau^2)}\,R_n(y)
}{
y-\tau^2
},
\quad n\in\mathbb N.
\]

\State Construct the monic orthogonal polynomial sequence
\[
(P_n)_{n\in\mathbb N}
\]
associated with \(\mathbf u\) by
\[
P_{2n}(x)=R_n(x^2),
\quad
P_{2n+1}(x)=(x-\tau)S_n(x^2),
\quad n\in\mathbb N.
\]

\State Return \((P_n)_{n\in\mathbb N}\), the alternating family generated
from \((R_n)_{n\in\mathbb N}\).
\end{algorithmic}
\end{tcolorbox}
\end{algorithm}

\begin{algorithm}[t]
\caption{Recurrence recognition and recovery}
\label{alg:alternating-recognition}
\begin{tcolorbox}[algobox]
\begin{algorithmic}[1]
\Require A regular monic polynomial sequence satisfying
\[
P_{k+1}(x)=(x-B_k)P_k(x)-C_kP_{k-1}(x),
\quad k\in\mathbb N,
\]
where \(P_{-1}=0\), \(P_0=1\), \(C_0=0\), and
\[
C_k\neq0,
\quad k\in\mathbb N^\times.
\]
\State Set \(\tau=B_0\).
\State Verify the necessary and sufficient identities
\[
B_{2n}=\tau,
\quad
B_{2n+1}=-\tau,
\quad n\in\mathbb N.
\]
\State If either identity fails, return that the sequence is not a direct
alternating output.
\State Otherwise recover the input recurrence coefficients from
\[
\begin{aligned}
\beta_0&=\tau^2+C_1,\\[7pt]
\beta_n&=\tau^2+C_{2n}+C_{2n+1},
\quad
\gamma_n=C_{2n-1}C_{2n},
\quad n\in\mathbb N^\times.
\end{aligned}
\]
\State Return \(\tau\) and the recovered input recurrence.
\end{algorithmic}
\end{tcolorbox}
\end{algorithm}

It was shown in \cite[Examples~8.11 and~8.13]{C26} that the so-called
complementary Bannai--Ito polynomials and the dual \((-1)\)-Hahn polynomials
are particular instances of the preceding construction, the latter through
its finite version. Several families in the literature admit the same
interpretation, although in some cases a further spectral transformation is
required to reveal it.

A Geronimus inversion of a functional \(\mathbf w\) at \(\mu\in\mathbb C\)
means a regular solution \(\mathbf q\) of
\[
(x-\mu)\mathbf q=\mathbf w.
\]
The solutions of the functional equation differ by multiples of the
evaluation functional \(\boldsymbol\delta_\mu\), where
\[
\langle\boldsymbol\delta_\mu,p\rangle=p(\mu),
\quad p\in\mathcal P.
\]
Thus one additional moment, or the published connection coefficients, selects
the required member, and regularity must be checked. If \(\mathbf w\) is the
direct output corresponding to \(\tau\), then, for every \(p\in\mathcal P\),
\[
\langle\mathbf q,(x-\mu)(x-\tau)p(x^2)\rangle
=\langle\mathbf w,(x-\tau)p(x^2)\rangle
=0.
\]
Consequently, \cite[Proposition~8.3]{C26} gives
\[
\mathbf D\bigl((x-\mu)(x-\tau)\mathbf q\bigr)=0.
\]
Hence \(\mathbf q\) is classical with
\[
\phi(x)=(x-\mu)(x-\tau),
\quad
\psi=0.
\]

\section{Examples}
The following calculations identify the three families treated in
\cite{VZ12,TVZ12,PVZ24} within the preceding mechanism.

\subsection{Big \((-1)\)-Jacobi polynomials}
The construction in \cite{VZ12} is obtained from
Algorithm~\ref{alg:alternating-pullback} followed by one regular Geronimus
inversion. First, one applies the alternating construction to a suitable
regular Jacobi family with \(\tau=1\). This gives the intermediate family
denoted \((R_n)_{n\in\mathbb N}\) in \cite[Section~3]{VZ12}. One then applies
the inverse connection in \cite[(4.1), (4.2)]{VZ12}, whose Geronimus
interpretation is given in \cite[(4.3), (4.4)]{VZ12}. The resulting family is
the big \((-1)\)-Jacobi family, and the explicit formulae
\cite[(2.24), (2.25)]{VZ12} follow by direct calculation.
In order to avoid a collision of parameters, our \((\lambda,c,d)\) correspond
to \((c,\alpha,\beta)\) in \cite{VZ12}.

Throughout this example we assume
\[
0<\lambda<1,
\]
and choose \(a,b\in\mathbb C\) so that the Jacobi functional below is regular;
see \cite[Table~1]{CG26a}. We assume, for every index under consideration,
that every displayed denominator is non-zero, that all monic normalising
factors are well defined and non-zero, and that the specified Geronimus
preimage is regular. No positivity assumption is imposed. The parameters
\(c,d\in\mathbb C\) used
later are given by
\[
a=\frac{c-1}{2},
\quad
b=\frac{d+1}{2}.
\]
We begin with the monic shifted Jacobi polynomial sequence
\[
\bigl(\widetilde{R}_n\bigr)_{n\in\mathbb N},
\]
given by
\[
\widetilde{R}_n(y)
=
(-1)^n\frac{(a+1)_n}{(n+a+b+1)_n}
\,{}_2F_1\!\left(
\begin{matrix}
-n,\ n+a+b+1\\[4pt]
a+1
\end{matrix}
\,;\, y
\right),
\quad n\in\mathbb N.
\]
Here \((z)_n=z(z+1)\cdots(z+n-1)\), with \((z)_0=1\), denotes the rising
factorial.
As in \cite[(3.13)]{VZ12}, we perform the affine change of variable
\[
y=\frac{1-t}{1-\lambda^2},
\]
and define
\[
R_n(t)=(\lambda^2-1)^n\widetilde{R}_n\!\left(\frac{1-t}{1-\lambda^2}\right),
\quad n\in\mathbb N.
\]
Since
\[
(\lambda^2-1)^n(-1)^n=(1-\lambda^2)^n,
\quad n\in\mathbb N,
\]
this becomes
\[
R_n(t)
=
(1-\lambda^2)^n\frac{(a+1)_n}{(n+a+b+1)_n}
\,{}_2F_1\!\left(
\begin{matrix}
-n,\ n+a+b+1\\[4pt]
a+1
\end{matrix}
\,;\,
\frac{1-t}{1-\lambda^2}
\right),
\quad n\in\mathbb N.
\]
This is the monic Jacobi-type family used in \cite[(3.13)]{VZ12}. We now choose
\[
\tau=1.
\]
The common factor \(x-1\) in the odd subsequence and the quadratic
transformation point \(1=\tau^2\) force this value.
To apply Algorithm~\ref{alg:alternating-pullback}, we must verify that
\[
R_n(1)\neq0
\]
for every \(n\in\mathbb N\). But
\[
R_n(1)
=
(1-\lambda^2)^n\frac{(a+1)_n}{(n+a+b+1)_n},
\]
because the hypergeometric term is evaluated at \(0\). The regularity and
non-vanishing assumptions above make this value non-zero. Hence the hypotheses
of Algorithm~\ref{alg:alternating-pullback} are satisfied.

For \(n\in\mathbb N\), set
\[
S_n(t)=
\frac{
R_{n+1}(t)-\dfrac{R_{n+1}(1)}{R_n(1)}R_n(t)
}{
t-1
}.
\]
This is the companion family introduced in \cite[(3.14)]{VZ12}. From the
explicit expression for \(R_n(1)\), we obtain, for \(n\in\mathbb N\),
\[
\frac{R_{n+1}(1)}{R_n(1)}
=
(1-\lambda^2)\,
\frac{(a+n+1)(a+b+n+1)}
     {(2n+a+b+1)(2n+a+b+2)}.
\]
This is the relabelled form of the coefficient in \cite[(3.16)]{VZ12}.
Substitution in the Christoffel formula gives, for \(n\in\mathbb N\),
\[
S_n(t)
=
(1-\lambda^2)^n\frac{(a+2)_n}{(n+a+b+2)_n}
\,{}_2F_1\!\left(
\begin{matrix}
-n,\ n+a+b+2\\[4pt]
a+2
\end{matrix}
\,;\,
\frac{1-t}{1-\lambda^2}
\right),
\]
which, after relabelling the parameters, is \cite[(3.15)]{VZ12}. We now apply
Algorithm~\ref{alg:alternating-pullback}. For \(m\in\mathbb N\), define
\[
P_{2m}(x)=R_m(x^2),
\quad
P_{2m+1}(x)=(x-1)S_m(x^2).
\]
We now rewrite these formulae in the parameters used in
\cite[Section~4]{VZ12}, namely
\[
a=\frac{c-1}{2},
\quad
b=\frac{d+1}{2},
\quad
z=\frac{1-x^2}{1-\lambda^2}.
\]
Then, for \(m\in\mathbb N\),
\begin{align*}
P_{2m}(x)
&=
(1-\lambda^2)^m
\frac{\left(\dps \frac{c+1}{2}\right)_m}
     {\left(\dps \frac{2m+c+d+2}{2}\right)_m}
\,{}_2F_1\!\left(
\begin{matrix}
-m,\ \dps\frac{2m+c+d+2}{2}\\[7pt]
\dps \frac{c+1}{2}
\end{matrix}
\,;\, z
\right)\\[7pt]
P_{2m+1}(x)
&=
-(1-x)
(1-\lambda^2)^m
\frac{\left(\dps \frac{c+3}{2}\right)_m}
     {\left(\dps \frac{2m+c+d+4}{2}\right)_m}
\,{}_2F_1\!\left(
\begin{matrix}
-m,\ \dps \frac{2m+c+d+4}{2}\\[7pt]
\dps \frac{c+3}{2}
\end{matrix}
\,;\, z
\right).
\end{align*}
These are the formulae in \cite[(3.21), (3.22)]{VZ12}, after relabelling the
parameters.

No additional structural ingredient is required at this stage: the preceding
calculation is a direct application of
Algorithm~\ref{alg:alternating-pullback}. It produces the family
\((P_n)_{n\in\mathbb N}\), denoted by \((R_n)_{n\in\mathbb N}\) in
\cite[Section~3]{VZ12}. This intermediate family must not be identified with
the big \((-1)\)-Jacobi polynomials. The latter are obtained from it by one
Geronimus inversion. At the level of functionals, let \(\mathbf w\) be the
functional for \((P_n)_{n\in\mathbb N}\), and let \(\mathbf u\) be that for
\((P_n^{(-1)})_{n\in\mathbb N}\), with \(\mathbf w\) normalised so that the
transformation has unit coefficient. In \cite[Section~4]{VZ12}, the Geronimus
point is \(\mu=-c\); under the parameter dictionary above, this is
\(\mu=-\lambda\); see \cite[(4.2)--(4.4)]{VZ12}. Hence
\[
(x+\lambda)\mathbf u=\mathbf w.
\]
Since the direct output has \(\tau=1\), \cite[Proposition~8.3]{C26} gives
\[
\mathbf D\bigl((x+\lambda)(x-1)\mathbf u\bigr)=0.
\]
Thus the structural pair is
\[
(\phi,\psi)=\bigl((x+\lambda)(x-1),0\bigr).
\]
The published connection coefficients select the required regular solution of
the inversion. In the present case, \cite[(4.1), (4.2)]{VZ12} writes the big
\((-1)\)-Jacobi polynomials as
\[
P_n^{(-1)}(x)=P_n(x)-g_nP_{n-1}(x),
\quad n\in\mathbb N^\times,
\]
where
\[
g_n=
\begin{cases}
\dfrac{(1-\lambda)n}{2n+c+d}, & n \text{ even},\\[11pt]
-\dfrac{(1+\lambda)(n+c)}{2n+c+d}, & n \text{ odd}
\end{cases}
\quad n\in\mathbb N^\times.
\]
We now check, explicitly and separately in the even and odd cases, that this
formula gives \cite[(2.24), (2.25)]{VZ12}.

\medskip
\noindent
\textbf{Even case.}
The case \(m=0\) is immediate, since \(P_0^{(-1)}(x)=1\). Let
\(m\in\mathbb N^\times\) and set \(n=2m\). Then
\[
P_{2m}^{(-1)}(x)=P_{2m}(x)-g_{2m}P_{2m-1}(x),
\]
with
\[
g_{2m}=\frac{(1-\lambda)2m}{4m+c+d}.
\]
Hence
\begin{align*}
P_{2m}^{(-1)}(x)
&=
(1-\lambda^2)^m
\frac{\dps \left(\frac{c+1}{2}\right)_m}
     {\left(\dps \frac{2m+c+d+2}{2}\right)_m}
\,{}_2F_1\!\left(
\begin{matrix}
-m,\ \dps\frac{2m+c+d+2}{2}\\[4pt]
\dps \frac{c+1}{2}
\end{matrix}
\,;\, z
\right)
\\[7pt]
&\quad
+
\frac{(1-\lambda)2m}{4m+c+d}
(1-x)
(1-\lambda^2)^{m-1}
\frac{\left(\dps\frac{c+3}{2}\right)_{m-1}}
     {\left(\dps\frac{2m+c+d+2}{2}\right)_{m-1}}
\\[7pt]
&\quad\times
{}_2F_1\!\left(
\begin{matrix}
1-m,\ \dps\frac{2m+c+d+2}{2}\\[7pt]
\dps\frac{c+3}{2}
\end{matrix}
\,;\, z
\right).
\end{align*}
We now simplify the coefficient in the second term. First,
\[
\left(\frac{c+1}{2}\right)_m
=
\frac{c+1}{2}\left(\frac{c+3}{2}\right)_{m-1}.
\]
Secondly,
\[
\left(\frac{2m+c+d+2}{2}\right)_m
=
\left(\frac{2m+c+d+2}{2}\right)_{m-1}
\left(\frac{4m+c+d}{2}\right).
\]
Therefore
\begin{align*}
&(1-\lambda^2)^m
\frac{\left(\dps\frac{c+1}{2}\right)_m}
     {\left(\dps\frac{2m+c+d+2}{2}\right)_m}
\\[7pt]
&=
(1-\lambda^2)^m
\frac{(c+1)\left(\dps \frac{c+3}{2}\right)_{m-1}}
     {(4m+c+d)\left(\dps \frac{2m+c+d+2}{2}\right)_{m-1}}.
\end{align*}
Hence
\[
(1-\lambda^2)^{m-1}
\frac{\left(\dps \frac{c+3}{2}\right)_{m-1}}
     {\left(\dps \frac{2m+c+d+2}{2}\right)_{m-1}}
=
\frac{4m+c+d}{(1-\lambda^2)(c+1)}
(1-\lambda^2)^m
\frac{\left(\dps \frac{c+1}{2}\right)_m}
     {\left(\dps \frac{2m+c+d+2}{2}\right)_m}.
\]
Multiplying by \(g_{2m}\), we obtain
\begin{align*}
&\frac{(1-\lambda)2m}{4m+c+d}
(1-\lambda^2)^{m-1}
\frac{\left(\dps\frac{c+3}{2}\right)_{m-1}}
     {\left(\dps\frac{2m+c+d+2}{2}\right)_{m-1}}\\[7pt]
     &\quad=
\frac{2m}{(1+\lambda)(c+1)}
(1-\lambda^2)^m
\frac{\left(\dps\frac{c+1}{2}\right)_m}
     {\left(\dps\frac{2m+c+d+2}{2}\right)_m}.
\end{align*}
Thus
\begin{align*}
\kappa^{-1}_{2m}\,P_{2m}^{(-1)}(x)
&=
{}_2F_1\!\left(
\begin{matrix}
-m,\ \dps\frac{2m+c+d+2}{2}\\[4pt]
\dps\frac{c+1}{2}
\end{matrix}
\,;\, z
\right)\\[7pt]
&\quad +
\frac{2m(1-x)}{(1+\lambda)(c+1)}
{}_2F_1\!\left(
\begin{matrix}
1-m,\ \dps\frac{2m+c+d+2}{2}\\[4pt]
\dps\frac{c+3}{2}
\end{matrix}
\,;\, z
\right),
\end{align*}
where
\[
\kappa_{2m}
=
(1-\lambda^2)^m
\frac{\left(\dps\frac{c+1}{2}\right)_m}
     {\left(\dps\frac{2m+c+d+2}{2}\right)_m}.
\]
Replacing \(2m\) by \(n\), we obtain \cite[(2.24), (2.26)]{VZ12}.

\medskip
\noindent
\textbf{Odd case.}
Let \(m\in\mathbb N\). Proceeding as in the preceding case, one obtains
\begin{align*}
\kappa^{-1}_{2m+1}\,P_{2m+1}^{(-1)}(x)
&=
{}_2F_1\!\left(
\begin{matrix}
-m,\ \dps\frac{2m+c+d+2}{2}\\[4pt]
\dps\frac{c+1}{2}
\end{matrix}
\,;\, z
\right)\\[7pt]
&\quad -
\frac{(2m+c+d+2)(1-x)}{(1+\lambda)(c+1)}
{}_2F_1\!\left(
\begin{matrix}
-m,\ \dps\frac{2m+c+d+4}{2}\\[4pt]
\dps \frac{c+3}{2}
\end{matrix}
\,;\, z
\right),
\end{align*}
where
\[
\kappa_{2m+1}
=
(1+\lambda)(1-\lambda^2)^m
\frac{\left(\dps\frac{c+1}{2}\right)_{m+1}}
     {\left(\dps\frac{2m+c+d+2}{2}\right)_{m+1}}.
\]
Replacing \(2m+1\) by \(n\), this is \cite[(2.25), (2.26)]{VZ12}.
Thus the example is not an isolated construction, but a concrete instance of
the general alternating mechanism.

\subsection{Bannai--Ito polynomials}

The parameters \(a,b,c,d\) in this subsection are independent of those in the
preceding example. They correspond, respectively, to
\((\rho_1,\rho_2,r_1,r_2)\) in \cite{TVZ12}. We assume that all the
displayed denominators are non-zero, that the monic normalising factors below
are well defined and non-zero, and that both the complementary Bannai--Ito
functional and its specified Geronimus preimage are regular.

We now turn to the Bannai--Ito case. Here the first alternating step need not
be reconstructed from scratch, since that part has already been carried out in
\cite[Example~8.11]{C26}. In the notation adopted here, the output of
Algorithm~\ref{alg:alternating-pullback} is the intermediate family
\[
(P_n)_{n\in\mathbb N},
\]
which coincides with the complementary Bannai--Ito family. The point to
stress, as in the big \((-1)\)-Jacobi case, is that the Bannai--Ito
polynomials do not coincide with this first output. Rather, they are obtained
from it by one further inverse step. In \cite[Section~5]{TVZ12}, the
complementary Bannai--Ito polynomials are described as Christoffel transforms
of the Bannai--Ito polynomials. To avoid a collision with the recurrence
coefficients \(B_k\), write \(I_n\) for the monic Bannai--Ito polynomial of
degree \(n\). The inverse relation is
\[
I_n(x)=P_n(x)-g_nP_{n-1}(x),
\quad n\in\mathbb N^\times,
\]
where
\[
g_n=
\begin{cases}
\displaystyle
-\frac{n(n-2c-2d)}{4(n-c-d+a+b)},
& n \text{ even},\\[16pt]
\displaystyle
-\frac{(n-2d+2b)(n-2c+2b)}{4(n-c-d+a+b)},
& n \text{ odd}
\end{cases}
\quad n\in\mathbb N^\times.
\]
This is the Geronimus inverse \cite[(5.4)]{TVZ12} of the Christoffel transform
\cite[(5.11)]{TVZ12}, with \(g_n=C_n\) as in \cite[(4.42)]{TVZ12},
rewritten in the present notation. The intermediate family
\((P_n)_{n\in\mathbb N}\) is the complementary Bannai--Ito family, while
\((I_n)_{n\in\mathbb N}\) is the Bannai--Ito family itself.

We now write explicitly the intermediate family produced by the algorithm. In
the notation adopted here, the complementary Bannai--Ito family takes the form
\[
P_{2n}(x)=R_n(x^2),
\quad
P_{2n+1}(x)=(x-b)S_n(x^2),
\quad n\in\mathbb N,
\]
where
\[
R_n(x^2)
=
\kappa_n^{(1)}
\,{}_4F_3\!\left(
\begin{matrix}
-n,\ n+a+b-c-d+1,\ b+x,\ b-x\\[4pt]
a+b+1,\ b-c+\dps\frac12,\ b-d+\dps\frac12
\end{matrix}
\,;\,1
\right),
\quad n\in\mathbb N,
\]
and
\[
S_n(x^2)
=
\kappa_n^{(2)}
\,{}_4F_3\!\left(
\begin{matrix}
-n,\ n+a+b-c-d+2,\ b+1+x,\ b+1-x\\[4pt]
a+b+2,\ b-c+\dps \frac32,\ b-d+\dps \frac32
\end{matrix}
\,;\,1
\right),
\quad n\in\mathbb N.
\]
The monic normalising factors are
\[
\kappa_n^{(1)}
=
\frac{
(1+a+b)_n\left(b-c+\dps\frac12\right)_n\left(b-d+\dps\frac12\right)_n
}{
(n+a+b-c-d+1)_n
},
\quad n\in\mathbb N,
\]
\[
\kappa_n^{(2)}
=
\frac{
(2+a+b)_n\left(b-c+\dps\frac32\right)_n\left(b-d+\dps \frac32\right)_n
}{
(n+a+b-c-d+2)_n
},
\quad n\in\mathbb N.
\]
The common factor \(x-b\) in the odd subsequence identifies the direct output
parameter as \(\tau=b\). Let \(\mathbf w\) and \(\mathbf u\) denote the
complementary Bannai--Ito and Bannai--Ito functionals, respectively, with
\(\mathbf w\) normalised so that the transformation has unit coefficient. By
\cite[(5.5), (5.6)]{TVZ12}, with the transformation point in
\cite[(5.11)]{TVZ12},
\[
(x-a)\mathbf u=\mathbf w.
\]
Consequently, \cite[Proposition~8.3]{C26} gives
\[
\mathbf D\bigl((x-a)(x-b)\mathbf u\bigr)=0,
\]
and the structural pair is
\[
(\phi,\psi)=\bigl((x-a)(x-b),0\bigr).
\]
Taking the leading coefficients in \cite[(5.18), (5.19)]{TVZ12} gives the
displayed monic normalising factors, while \cite[(5.20)]{TVZ12} gives the
parity splitting. In particular, the second factor involving \(r_1\) in the
printed normalising constants following those formulae must involve \(r_2\),
as both the leading-coefficient calculation and the symmetry in
\((r_1,r_2)\) require. We now verify, in direct parallel with the preceding
example, that the Geronimus formula produces the Bannai--Ito family from this
intermediate output.

\medskip
\noindent
\textbf{Even case.}
The case \(m=0\) is immediate, since \(I_0(x)=P_0(x)=1\). Let
\(m\in\mathbb N^\times\) and set \(n=2m\). Then
\[
I_{2m}(x)=P_{2m}(x)-g_{2m}P_{2m-1}(x).
\]
Using the formulae above, we obtain
\[
P_{2m}(x)=R_m(x^2),
\quad
P_{2m-1}(x)=(x-b)S_{m-1}(x^2),
\]
and therefore
\[
I_{2m}(x)=R_m(x^2)-g_{2m}(x-b)S_{m-1}(x^2).
\]
Since
\[
g_{2m}
=
-\frac{2m(2m-2c-2d)}{4(2m-c-d+a+b)}
=
-\frac{m(m-c-d)}{2m-c-d+a+b},
\]
it follows that
\begin{align*}
I_{2m}(x)
&=
\kappa_m^{(1)}
\,{}_4F_3\!\left(
\begin{matrix}
-m,\ m+a+b-c-d+1,\ b+x,\ b-x\\[4pt]
a+b+1,\ b-c+\dps\frac12,\ b-d+\dps\frac12
\end{matrix}
\,;\,1
\right)
\\[7pt]
&\quad
+
\frac{m(m-c-d)}{2m-c-d+a+b}
(x-b)\kappa_{m-1}^{(2)}\\[7pt]
&\quad\times {}_4F_3\!\left(
\begin{matrix}
1-m,\ m+a+b-c-d+1,\ b+1+x,\ b+1-x\\[4pt]
a+b+2,\ b-c+\frac32,\ b-d+\dps\frac32
\end{matrix}
\,;\,1
\right).
\end{align*}
Replacing \(2m\) by \(n\), this is what follows from
\cite[(5.20), (5.21)]{TVZ12} after substituting the explicit expressions in
\cite[(5.18), (5.19)]{TVZ12}.

\medskip
\noindent
\textbf{Odd case.}
Let \(m\in\mathbb N\). Proceeding as in the preceding case, one obtains
\begin{align*}
I_{2m+1}(x)
&=
(x-b)\kappa_m^{(2)}
\,{}_4F_3\!\left(
\begin{matrix}
-m,\ m+a+b-c-d+2,\ b+1+x,\ b+1-x\\[4pt]
a+b+2,\ b-c+\dps\frac32,\ b-d+\dps\frac32
\end{matrix}
\,;\,1
\right)
\\[7pt]
&\quad
+
\frac{(2m+1-2d+2b)(2m+1-2c+2b)}
     {4(2m+1-c-d+a+b)}
\kappa_m^{(1)}\\[7pt]
&\quad\times{}_4F_3\!\left(
\begin{matrix}
-m,\ m+a+b-c-d+1,\ b+x,\ b-x\\[4pt]
a+b+1,\ b-c+\dps\frac12,\ b-d+\dps\frac12
\end{matrix}
\,;\,1
\right).
\end{align*}
Replacing \(2m+1\) by \(n\), this is what follows from
\cite[(5.20), (5.21)]{TVZ12} after substituting the explicit expressions in
\cite[(5.18), (5.19)]{TVZ12}. Thus the Bannai--Ito family is recovered from
the output of Algorithm~\ref{alg:alternating-pullback} by one further
Geronimus transform, as in the big \((-1)\)-Jacobi case.

\subsection{The \((-1)\)-Meixner--Pollaczek polynomials}

We finish with the \((-1)\)-Meixner--Pollaczek polynomials, since in this case
the reduction is immediate. The input for Algorithm~\ref{alg:alternating-pullback}
is the Laguerre functional shifted by \(\gamma^2\). Fix
\(\alpha,\gamma\in\mathbb C\), and assume that
\[
\alpha\neq-\frac12-j,
\quad j\in\mathbb N.
\]
Let
\(\mathbf v\) be the regular functional whose monic orthogonal polynomial
sequence is
\[
R_n(y)
=
(-1)^n\left(\alpha+\frac12\right)_n
\,{}_1F_1\!\left(
\begin{matrix}
-n\\[2pt]
\alpha+\frac12
\end{matrix}
;\,y-\gamma^2
\right),
\quad n\in\mathbb N.
\]
Thus \((R_n)_{n\in\mathbb N}\) is simply the monic Laguerre family in the variable
\(y-\gamma^2\), with parameter \(\alpha-\frac12\).

We now apply Algorithm~\ref{alg:alternating-pullback} with
\[
\tau=\gamma.
\]
The common factor \(x-\gamma\) in the published odd subsequence identifies
this value directly.
The required non-vanishing condition is
\[
R_n(\gamma^2)
=
(-1)^n\left(\alpha+\frac12\right)_n\neq0,
\quad n\in\mathbb N,
\]
by the displayed assumption. The
Christoffel transform of \(\mathbf v\) at the point \(\gamma^2\) has monic
orthogonal polynomial sequence
\[
S_n(y)
=
(-1)^n\left(\alpha+\frac32\right)_n
\,{}_1F_1\!\left(
\begin{matrix}
-n\\[2pt]
\alpha+\frac32
\end{matrix}
;\,y-\gamma^2
\right),
\quad n\in\mathbb N.
\]
Therefore the algorithm gives
\[
P_{2n}(x)=R_n(x^2),
\quad
P_{2n+1}(x)=(x-\gamma)S_n(x^2),
\quad n\in\mathbb N.
\]
That is,
\[
P_{2n}(x)
=
(-1)^n\left(\alpha+\frac12\right)_n
\,{}_1F_1\!\left(
\begin{matrix}
-n\\[2pt]
\alpha+\frac12
\end{matrix}
;\,x^2-\gamma^2
\right),
\quad n\in\mathbb N,
\]
and
\[
P_{2n+1}(x)
=
(-1)^n\left(\alpha+\frac32\right)_n(x-\gamma)
\,{}_1F_1\!\left(
\begin{matrix}
-n\\[2pt]
\alpha+\frac32
\end{matrix}
;\,x^2-\gamma^2
\right),
\quad n\in\mathbb N.
\]
These are the \((-1)\)-Meixner--Pollaczek polynomials displayed in
\cite[(A143) and~(A144)]{PVZ24}. This is the direct structural pair
\[
(\phi,\psi)=\bigl(0,x-\gamma\bigr).
\]
Thus there is no separate construction to be discovered here:
the shifted Laguerre functional and \(\tau=\gamma\) give the published family
directly through Algorithm~\ref{alg:alternating-pullback}.

\section{Conclusion}
The preceding construction does not add three more names to the list of
\((-1)\)-classical families. It identifies the mechanism already governing
them. For every fixed \(\tau\in\mathbb C\), each polynomial has the unique
decomposition
\[
p(x)=p_0(x^2)+(x-\tau)p_1(x^2),
\]
and, together with the annihilation condition
\[
\langle\mathbf u,(x-\tau)p(x^2)\rangle=0,
\quad p\in\mathcal P,
\]
this gives the direct reconstruction. Regularity is precisely the additional
issue: once the quadratic component is regular, \cite[Theorem~8.6]{C26} shows
that the output is regular precisely when
\[
R_n(\tau^2)\neq0,
\quad n\in\mathbb N.
\]
The resulting polynomials are
\[
P_{2n}(x)=R_n(x^2),
\quad
P_{2n+1}(x)=(x-\tau)S_n(x^2),
\quad n\in\mathbb N.
\]
The big \((-1)\)-Jacobi and Bannai--Ito families are the regular one-step
Geronimus inversions of such direct outputs; the
\((-1)\)-Meixner--Pollaczek family is a direct output. Their structural pairs
therefore place all three families within
Definition~\ref{def:classical-normalised-alternating}.

Figure~A1 of \cite{PVZ24} records limits, specialisations, and spectral
transformations between named families. It contains the big \((-1)\)-Jacobi
and \((-1)\)-Meixner--Pollaczek families, together with the continuous
Bannai--Ito family. The latter is obtained from the non-truncated
Bannai--Ito family by an imaginary change of variable and complex
parameters \cite[Section~5]{PVZ24}. In any event, the chart does not encode either the functional
annihilation condition or the quadratic reconstruction above. It is therefore
a genealogical chart, not a classification at the level of regular
functionals. The distinction is decisive: the chart describes how the names
are connected, whereas the alternating construction explains why the same
even--odd structure reappears. The structure precedes the hypergeometric
realisation.

\section*{Acknowledgements}
The authors acknowledge financial support from the Centre for Mathematics of
the University of Coimbra (CMUC), funded by the Portuguese Foundation for
Science and Technology (FCT), under the projects UID/00324/2025
(\url{https://doi.org/10.54499/UID/00324/2025}) and UID/PRR/00324/2025.
The first author acknowledges financial support from the FCT under the grant
\url{https://doi.org/10.54499/2022.00143.CEECIND/CP1714/CT0002}.
The second author acknowledges financial support from the FCT under the grant
UI.BD.154694.2023.

\section*{Data availability statement}

No data were generated or analysed in the course of this study.

\section*{Conflict of interest statement}

The authors declare that they have no conflict of interest.

\end{document}